\documentclass{amsart}

\usepackage{amsmath,amssymb,amsthm}
\usepackage{enumitem}
\usepackage{fixltx2e}
\usepackage{mathtools}
\usepackage{graphicx}
\usepackage{tikz-cd}
\usepackage[textsize=small]{todonotes}
\usepackage{xparse}
\usepackage{xspace}
\usepackage{mathtools}
\usepackage{dsfont}
\usepackage[all]{xy}
\usepackage{color}
\usepackage{verbatim}
\usepackage[colorlinks,urlcolor=blue,linkcolor=blue]{hyperref}
\usepackage{mathrsfs}
\usepackage{microtype}
\usepackage{a4wide}
\usepackage{cleveref}

\usepackage{longtable} % for 'longtable' environment
\usepackage{pdflscape} % for 'landscape' environment
\usepackage{array}

\makeatletter
\def\Cline#1#2{\@Cline#1#2\@nil}
\def\@Cline#1-#2#3\@nil{%
	\omit
	\@multicnt#1%
	\advance\@multispan\m@ne
	\ifnum\@multicnt=\@ne\@firstofone{&\omit}\fi
	\@multicnt#2%
	\advance\@multicnt-#1%
	\advance\@multispan\@ne
	\leaders\hrule\@height#3\hfill
	\cr}
\makeatother

\numberwithin{equation}{section}
\mathtoolsset{showonlyrefs=true}

\def\mod{\operatorname{mod}}
\def\Supp{\mathop{\text{\upshape Supp}}}

\def\GL{\operatorname {GL}}

\newcommand{\Z}{{\mathbb Z}}

\newcommand{\Q}{{\mathbb Q}}

\newcommand{\C}{{\mathbb C}}
\newcommand{\U}{{\mathcal U}}

\newcommand{\A}{\mathbb A}

\newtheorem{lemma}{Lemma}[section]
\newtheorem{proposition}[lemma]{Proposition}
\newtheorem{theorem}[lemma]{Theorem}
\newtheorem{corollary}[lemma]{Corollary}

\theoremstyle{definition}

\usepackage{xcolor}

\mathchardef\mhyphen="2D

\newcounter{todocounter}
\DeclareDocumentCommand\addreference{g}{\stepcounter{todocounter}\todo[color = blue!30, fancyline]{\thetodocounter. Add reference\IfNoValueF{#1}{: #1}}\xspace}
\DeclareDocumentCommand\checkthis{g}{\stepcounter{todocounter}\todo[color = red!50, fancyline]{\thetodocounter. Check this\IfNoValueF{#1}{: #1}}\xspace}
\DeclareDocumentCommand\fixthis{g}{\stepcounter{todocounter}\todo[color = orange!50, fancyline]{\thetodocounter. Fix this\IfNoValueF{#1}{: #1}}\xspace}
\DeclareDocumentCommand\expand{g}{\stepcounter{todocounter}\todo[color = green!50, fancyline]{\thetodocounter. Expand\IfNoValueF{#1}{: #1}}\xspace}

\title[]{Free products in the unit group of the integral group ring of a finite group}
\author{Geoffrey Janssens}
\author{Eric Jespers}
\author{Doryan Temmerman}

\address{Department of Mathematics, Vrije Universiteit Brussel,
Pleinlaan $2$, B-1050 Elsene, Belgium \newline E-mail address: {\tt geofjans@vub.ac.be},  {\tt efjesper@vub.ac.be},  {\tt dtemmerm@vub.ac.be} }
\begin{document}

\begin{abstract}
Let $G$ be a finite group and let $p$ be a prime. 
We continue  the search for  generic constructions of  free products and free monoids  in the unit group  $\mathcal{U}(\mathbb{Z}G)$ of the integral group ring $\Z G$.
For a nilpotent group  $G$ with a non-central element $g$ of order $p$,
explicit generic constructions are given of two periodic units $b_1$ and $b_2$ in
$\mathcal{U}(\mathbb{Z}G)$ such that $\langle b_1 , b_2\rangle =\langle b_1\rangle \star \langle b_2 \rangle \cong \mathbb{Z}_p
\star \mathbb{Z}_{p}$, a free product of two cyclic groups of prime order.
Moreover, if $G$ is  nilpotent of class $2$ and $g$ has order $p^n$, then also concrete generators for  free  products $\Z_{p^k} \star \Z_{p^m}$ are constructed (with $1\leq k,m\leq n $).  As an application, for finite nilpotent groups,  we obtain earlier results of Marciniak-Sehgal and Gon{\c{c}}alves-Passman. Further, for an arbitrary finite group $G$ we give generic constructions of  free monoids in $\mathcal{U}(\mathbb{Z}G)$ that generate an infinite solvable subgroup.
\end{abstract}

\maketitle

%\setcounter{tocdepth}{1}
%{%\hypersetup{linkcolor=black}
%\tableofcontents
%}

\newcommand\blfootnote[1]{%
  \begingroup
  \renewcommand\thefootnote{}\footnote{#1}%
  \addtocounter{footnote}{-1}%
  \endgroup
}

\blfootnote{The second author is supported by
Onderzoeksraad of Vrije Universiteit Brussel and Fonds voor Wetenschappelijk Onderzoek (Flanders).
%} 
%\blfootnote{
The first and third authors are supported by Fonds voor Wetenschappelijk Onderzoek (Flanders).
}

\section{Introduction}

The representation theory over some ring $R$ of a finite  group $G$ is nowadays studied through the modules over its group ring $RG$. In this perspective it is natural to ask whether $RG$ determines the underlying group $G$, i.e. if $RG \cong RH$ is then $G \cong H$? This question is called the isomorphism problem for group rings. In case $R = \mathbb{Z}$ the problem remained open till the surprising counterexample by Hertweck in 1999 \cite{Her}. However for many classes of groups the isomorphism problem has a positive answer over $\mathbb{Z}$. One main obstruction towards a positive solution is the lack of information on how rigid $G$ lies inside $\mathbb{Z}G$ or more precisely inside $\mathcal{U}(\mathbb{Z}G)$, the unit group of $\Z G$. Several conjectures have been made in this direction, such as the Zassenhaus conjectures, and consequently one started the search for generic constructions of elements in $\mathcal{U}(\mathbb{Z}G)$ and one investigates the algebraic structure of the group generated by these units.

Only few generic constructions are known. The most important are  the so called Bass units  and the bicylic units. With these elements at hand, it is a natural problem to determine  ``how large''  the group $B$ generated by the Bass and bicyclic units is compared to $\mathcal{U}(\mathbb{Z}G)$. Furthermore,  one would like to determine the  relations between these units. Jespers and Leal \cite{JesLeal} proved that for many finite groups $G$ the group $B$  is of finite index in $\mathcal{U}(\mathbb{Z}G)$; earlier results of this type were obtained by Ritter and Sehgal (see for example \cite{rit-seh}).  The groups $G$ excluded are those that have a non-commutative fixed point free image and  those for which  the rational group algebra  $\mathbb{Q}G$  does have an exceptional simple epimorphic image. 
The latter are  by definition the  non-commutative division algebras which are not a positive-definite quaternion algebra and matrix algebras $M_2(D)$ over a division algebra of the type   $\mathbb{Q}$,  $\mathbb{Q}(\sqrt{-d})$ or a quaternion algebra $\left( \frac{a, b}{\mathbb{Q}}\right)$.  Moreover, in \cite{EisKieVGel}, Eisele, Kiefer and Van Gelder reduced the number of exceptional cases by showing that the only cases that can occur as an epimorphic image of a rational groups algebra $\mathbb{Q}G$ are $d=1,2,3$ and $(a,b) = (-1,-1), (-1,-3), (-2,-5)$. For a state-of-the-art we refer to   \cite{EricAngel1}. 

In recent years there have been a lot of investigations on determining whether there are any non-trivial relations between two given units that are Bass units or bicyclic units.
It turns out that in many cases two such  elements generate a non-cyclic free group. 
In this context, a result of  Hartley and Pickel \cite{HarPic} states that $\mathcal{U}(\mathbb{Z}G)$ contains a non-cyclic  free group except  if $G$ is abelian or an Hamiltonian 2-group. Actually, it turns that these cases correspond with  $\mathcal{U}(\mathbb{Z}G)$ being abelian-by-finite, which on its turn is exactly the case when the unit group is solvable-by-finite (see \cite[Corollary 5.5.7]{EricAngel1}).

An explicit construction of a free subgroup of the unit group was given by Marciniak and Sehgal in 
 \cite{MarciniakSehgal}: it is shown that  any non-trivial bicyclic unit together with its image under the classical involution (which also is a bicylic unit)   generate a non-cyclic  free group. Since  then many more constructions of two bicyclic units, or two Bass units, or a Bass together with a bicyclic unit generating a free group have been discovered. For a  survey we refer to \cite{GonAngel,EricAngel2}. In \cite{GonPass}, Gon{\c{c}}alves and Passman 
 showed that $\mathcal{U}(\mathbb{Z}G)$  contains a free product $\mathbb{Z}_p \star  \mathbb{Z}$
 (with $p$ a prime number) if and only if $G$ contains a noncentral element of order $p$. 
Moreover, when this occurs, the $\mathbb{Z}_p$-part of the free product can be taken to be a suitable noncentral subgroup of $G$ of order $p$. 
The proof of this  result makes use of  earlier work of Passman \cite{Pass} on the existence in ${\rm PSL}_n(R)$ (with $R$ a commutative integral domain of characteristic zero)  of a free product $G \star \mathbb{Z}$ when  $G$ is a finite subgroup of ${\rm PSL}_n(R)$.
In the proofs of all these results, the element of infinite order is used in order to apply Tits alternative type techniques. We point out that this generator of the infinite cyclic part is only shown to exist, but no explicit constructions are obtained.

Note that if $p\neq 2$ then  a group contains $\mathbb{Z}_p \star  \mathbb{Z}$ if and only if it contains a group $\mathbb{Z}_p \star  \mathbb{Z}_p$. Hence, a natural problem is to give explicit generic constructions of units $b_1, b_2 \in \mathcal{U}(\mathbb{Z}G)$ such that $\langle b_1 , b_2 \rangle \cong \mathbb{Z}_p \star  \mathbb{Z}_p$.
As such, one also obtains a generic construction of a  unit $b=b_2 b_1  b_2$ such that $\langle b_1 , b \rangle \cong \mathbb{Z}_p \star  \mathbb{Z},$ provided $p\neq 2$.
So far, such a result has not been obtained mainly because of the lack of generic constructions of non-trivial torsion units.
However, recently in \cite{bob-free}, V. Bovdi introduced very interesting torsion units.

In this paper we make use of these  torsion units, we simply call them  Bovdi units, to give the first (generic) constructions of free products of two finite cyclic groups in $\mathcal{U}(\mathbb{Z}G)$ provided $G$ is a finite nilpotent group.
Moreover, for an arbitrary finite group $G$, we also deal with the problem of producing infinite solvable subgroups in $\mathcal{U}(\mathbb{Z}G)$ (that are not virtually nilpotent)  using non-trivial Bass and bicyclic units and construct  
generators of  a non-cyclic free submonoid in these groups.

Throughout the paper $G$ denotes a finite group and for a subset $H$ of $G$ we put 
 $\widetilde{H}= \sum_{h \in H} h\in \mathbb{Z}G $ and $\widehat{H} = \frac{1}{|H|} \widetilde{H}\in \mathbb{Q}G$.
 If $H$ is a subgroup then $\widehat{H}$ is an idempotent. In case $H=\langle g\rangle$ then we simply denote these
 elements by $\widetilde{g}$ and $\widehat{g}$ respectively. 
 The classical involution on a group algebra $KG$ is the $K$-linear map defined by $g\mapsto g^{-1}$, for $g\in G$.
 By $\omega$ we denote the augmentation map $K G \rightarrow K$. The free product of two groups $H$ and $K$ is denoted by 
$H \star K$. By $o(g)$ we denote the order of $g\in G$ and a cyclic group of order $n$ (possibly $n$ is infinite) we will denote by $C_n$.

\section{Preliminaries}\label{prelims}

In order to construct free monoids and free products of cyclic groups in $\mathcal{U}(\Z G)$, we will make use of   representations of the finite groups $G$ and hence of matrix epimorphic images of the complex group algebra $\mathbb{C}G$.  Therefore we begin with some technical lemmas on constructions of free products in matrix algebras. 
By $E_{ij}$ we denote the elementary matrix with $1$ in position (i,j) and zeroes elsewhere.
We start with three lemmas. The first and the third can  be found in \cite[pages 2,3]{EricAngel2} and the second
easily is verified.

\begin{lemma}\label{lemma_freepoint}
Suppose $z_1,z_2,w_1,w_2$ are non-zero elements in a field $F$. Then, there exists $U \in GL_2(F)$ such that $E_{12}(z_1)^U = E_{12}(w_1)$ and $E_{21}(z_2)^U = E_{21}(w_2)$ if and only if $z_1z_2 = w_1w_2$.
\end{lemma}
%\begin{proof}
%See \cite[pg. 3]{EricAngel2}.
%\end{proof}

\begin{lemma}\label{Lemma_rootsofunity}
Let $n$ be a positive integer and $\zeta_n= e^{\frac{2\pi i}{n}}$ is a complex primitive $n$-th root of unity.
If   $k$ is a positive integer such that  $k \neq 0 \mod n$ then  $| \sum_{i=0}^{k-1} \zeta_n^i| \geq 1$; and this inequality is strict if and only if $k \neq \pm 1 \mod n$.
\end{lemma}

%\begin{proof}
%This is an easy calculation with complex numbers.
%\end{proof}

A useful tool for proving that the group generated by two subgroups is actually a free product is the well-known Ping-Pong Lemma.

\begin{lemma}{\emph{(Ping-Pong Lemma).}}\label{Lemma_PingPong}
Suppose $G_1$ and $G_2$ are non-trivial subgroups of a group $G$	with $|G_1| > 2$. 
If $G$ acts on a set $P$ 
and $P_1$ and $P_2$ are two non-empty and distinct subsets of $P$ such that $g(P_i)\subseteq P_j$ for every non-trivial $g\in G_i$ and $\{ i,j \}=\{ 1,2\}$ then $\langle G_1, G_2 \rangle \cong G_1 \star G_2$.
%
%such that there are non-empty and different subsets $P_1$ and $P_2$ with the property that $g(P_i)\subseteq P_j$ 
%for $1\neq g \in G_i$ and $\{i,j\} = \{1,2\}$, then $\langle G_1, G_2 \rangle \cong G_1 \star G_2$.
\end{lemma}
%\begin{proof}
%See \cite[pg. 2]{EricAngel2}
%\end{proof}

This Ping-Pong Lemma allows to construct a free product of cyclic groups in $GL_n(\C)$ and thus we obtain the  following Sanov-like result.

\begin{proposition}\label{Sanov_cyclic}
Let $n$ and  $m$ be positive integers and let $\zeta_n=e^{\frac{2k_n\pi i}{n}}$ and $\zeta_m= e^{\frac{2k_m\pi i}{m}}$ be primitive complex roots of unity of order $n$ and $m$ respectively 
 ($k_n$ and $k_m$ are integers such that $1\leq k_n \leq n$, $1 \leq k_m \leq m$).
Put   $z_{k_n} = \sum_{j=0}^{k_n-1} e^{\frac{2j\pi i}{n}}$, $z_{k_m} = \sum_{j=0}^{k_m-1} e^{\frac{2j\pi i}{m}}$ 
and suppose $u,v\in \C$.  If either $n$ or $m$ is not $2$ and  $|uv| \geq 4|z_{k_n}z_{k_m}|$, or $n=m=2$ and $u$ or $v$ is non-zero, then the following isomorphisms hold $$\langle \left[\begin{array}{cc} \zeta_n & u \\ 0 & 1 \end{array}\right], \left[\begin{array}{cc} \zeta_m & 0 \\ v & 1 \end{array}\right] \rangle \cong C_{i_n} \star C_{i_m} \cong \langle \left[\begin{array}{cc} 1 & u \\ 0 & \zeta_n \end{array}\right], \left[\begin{array}{cc} 1 & 0 \\ v & \zeta_m \end{array}\right] \rangle, $$ where $i_t = t$ if $t \neq 1$ and $i_1 = \infty$.

\end{proposition}
\begin{proof}
It is clear that 
   $\langle \left[\begin{array}{cc} \zeta_n & u \\ 0 & 1 \end{array}\right], 
     \left[\begin{array}{cc} \zeta_m & 0 \\ v & 1 \end{array}\right] \rangle 
      \cong
       \langle \left[\begin{array}{cc} 1 & 0 \\ u & \zeta_n \end{array}\right], 
       \left[\begin{array}{cc} 1 & v \\ 0 & \zeta_m \end{array}\right] \rangle$.
So we only will  prove the first isomorphism.
%maybe switch the order of the proof. Make it a reduction to the case e^{2\pi i/n} instead of 2 parts.

Put  $A = \left[\begin{array}{cc} \zeta_n & u \\ 0 & 1 \end{array}\right] \quad \text{and} \quad B =\left[\begin{array}{cc} \zeta_m & 0 \\ v & 1 \end{array}\right]$. Clearly,  $\langle A \rangle \cong C_{i_n}$ and $\langle B \rangle \cong C_{i_m}$.

We first deal with the case that either $n$ or $m$ is not $2$ and  $|uv| \geq 4|z_{k_n}z_{k_m}|$. 
Suppose first that $k_n = k_m = 1$ and thus  $1=z_{k_n} = z_{k_m}$.
By \Cref{lemma_freepoint} we may assume 
 $|u| \geq 2|z_{k_n}| = 2$ and $|v|\geq 2|z_{k_m}| = 2$. The group $\langle A,B\rangle$ acts via 
 M\"obius transformations on $\C$. The M\"obius transformations  determined by  $A$ and $B$ are resprectively:
 $$\varphi_A : z \mapsto \zeta_n z + u \quad \mbox{ and } \quad  \varphi_B : z \mapsto \frac{\zeta_m z}{vz +1}. 
 $$ 
One easily verifies that for any positive integer $k$ the following equalities hold: 
$$\varphi_A^k(z) = \zeta_n^k z + \sum_{i=0}^{k-1} \zeta_n^i u \quad \mbox{ and } \quad   \varphi_B^k(z) = \frac{\zeta_m^k z}{\left(\sum_{i=0}^{k-1}\zeta_m^i\right) vz + 1}.$$ 
%We show this for $\varphi_A$, but the proof for $\varphi_B$ is similar. The formula for $k=1$ is trivial. By induction we calculate \begin{align*}
%\varphi^{k+1}_A(z) &= \varphi_A(\varphi_A^k(z)) = \varphi_A(\zeta_n^k z + \sum_{i=0}^{k-1} \zeta_n^i u)\\ &= \zeta_n(\zeta_n^k z + \sum_{i=0}^{k-1} \zeta_n^i u) + u = \zeta_n^{k+1}z + \sum_{i=0}^{(k+1)-1} \zeta_n^i u,
%\end{align*} as desired.
Consider in $\C$ the following regions $P_1 = \{z \in \C \mid |z| < 1\}$ and $P_2 = \{ z \in \C \mid |z| > 1\}$ of $\C$.
If $z \in P_1$ and $1\leq k < n$ (or $k$ a positive integer if $n = 1$), then by \Cref{Lemma_rootsofunity} $|\sum_{i=0}^{k-1} \zeta_n^i| \geq 1$ and hence
    \begin{eqnarray*}
         |\varphi_A^k(z)| &=& |\zeta_n^k z + \sum_{i=0}^{k-1} \zeta_n^i u| 
                                        \; \geq\;  |\sum_{i=0}^{k-1} \zeta_n^iu| - |\zeta^k_n z| 
                                        \; \geq \; |u| - |z| 
                                        \;  > \;1.
    \end{eqnarray*} 
 So,  $(\langle A \rangle \setminus \{ 1 \}) (P_1) \subseteq P_2$. 
 Similarly one shows that
    $(\langle B \rangle \setminus \{ 1 \}) (P_2) \subseteq P_1$. 
\Cref{Lemma_PingPong} (here we use that either $n$ or $m$ is not $2$) yields that $$\langle A,B \rangle = \langle A \rangle \star \langle B \rangle \cong C_{i_n} \star C_{i_m}. $$

Next we consider  arbitrary $\zeta_n$ and $\zeta_m$, for $n \neq 1 \neq m$ (and  we still work under the assumption  $n\neq 2$ or $m\neq2$). Clearly,
     $$\left[\begin{array}{cc} e^{\frac{2\pi i}{n}} & z^{-1}_{k_n}u \\ 0 & 1 \end{array}\right]^{k_n} = A 
         \quad \mbox{ and } \quad
         \left[\begin{array}{cc} e^{\frac{2\pi i}{m}} & 0\\ z^{-1}_{k_m}v  & 1 \end{array}\right]^{k_m} = B.$$ 
Because  $|z^{-1}_{k_n}uz^{-1}_{k_m}v| \geq \left| \frac{4|z_{k_n}z_{k_m}|}{z_{k_n}z_{k_m}}\right| = 4$, the conditions for the first part of the proof are satisfied for the matrices 
$\left[\begin{array}{cc} e^{\frac{2\pi i}{n}} & z^{-1}_{k_n}u \\ 0 & 1 \end{array}\right]$ and $\left[\begin{array}{cc} e^{\frac{2\pi i}{m}} & 0\\ z^{-1}_{k_m}v  & 1 \end{array}\right]$. Hence, they generate a group isomorphic to $C_{i_n} \star C_{i_m}$. 
Since $k_n$ and $k_m$ are coprime to $n$ and $m$ respectively,  $\langle A, B\rangle = \langle \left[\begin{array}{cc} e^{\frac{2\pi i}{n}} & z^{-1}_{k_n}u \\ 0 & 1 \end{array}\right],\left[\begin{array}{cc} e^{\frac{2\pi i}{m}} & 0\\ z^{-1}_{k_m}v  & 1 \end{array}\right]\rangle$ and thus the result follows in this case.

The case where $n=m=2$ has a more direct proof. Indeed, in this case 
 $A = \left[\begin{array}{cc} -1 & u \\ 0 & 1 \end{array}\right]$ and  
 $B =\left[\begin{array}{cc} -1 & 0 \\ v & 1 \end{array}\right]$, and these are  elements of  order $2$. 
 Also, $B^{-1} \left(AB\right) B = BABB = BA = (AB)^{-1}$, showing that $B$ acts on $AB$ by inversion. 
 So, $\langle A,B\rangle$ is the infinite dihedral group and thus isomorphic with $C_2 \star C_2$.
\end{proof}

The following lemma is  an easy consequence of the proposition but it is essential for the remainder of the paper.

\begin{lemma} \label{prop: generalised Sanov_cyclic}
We use  notations as in \Cref{Sanov_cyclic}.
%Suppose, for $n\geq 2$, that there exists a $\sigma \in S_n$ such that, applying this permutation to the bases of $\C^n$ transforms $A\in GL_n(\C)$ of order $p$ and $B \in GL_n(\C)$ of order $q$ into respectively a lower triangular matrix $T$ and $D + E_{12}(d)$, where $D$ is a diagonal matrix. 
%Assume $\zeta_n= e^{\frac{2k_n\pi}{n}}$ is a primitive $n$-th root of unity and $\zeta_m=e^{\frac{2k_m\pi}{m}}$ is a primitive $m$-the root of unity.
%$Let $r\geq 2$ and $\forall t \in \N_0\setminus \{1\}: i_t = t$, $i_1 = \infty$. 
Let $r$ be a positive integer different from $2$.
Suppose $A=(A_{ij}) \in \GL_r(\C)$  is a lower triangular matrix of order $i_n$ with $A_{11}=1$ and 
$A_{22}=\zeta_n$  and $B=D+dE_{12}\in \GL_r(\C)$  a matrix of order $i_m$ with 
$D$ a diagonal matrix, $D_{11}=1$ and $D_{12}=\zeta_m $. 
%$Let  $z_{k_n} = \frac{1-\zeta_n}{1-e^{\frac{2\pi i}{n}}} = 
%\sum_{j=0}^{k_n-1} e^{\frac{2j\pi i}{n}}$ and $z_{k_m} = \frac{1-\zeta_m}{1-e^{\frac{2\pi i}{m}}}=\sum_{j=0}^{k_m-1} e^{\frac{2j\pi i}{m}}$. 
%If the first two diagonal elements of $T$ (respectively $D$) are $1$ and $\zeta_p$ (respectively $\zeta_q$) and
If  $|A_{21}d| \geq 4|z_{k_n}z_{k_m}|$, 
then $\langle A,B \rangle \cong C_{i_n} \star C_{i_m}$.

\end{lemma}
\begin{proof}
Let  $L_r=\{ (a_{ij}) \in GL_r(\C) \mid a_{ij} = 0 \text{ if } j>i \text{ and } j>2\}$, a subgroup  of $GL_r(\C)$. 
%The product formula $(AB)_{ij} = \displaystyle\sum\limits_{k=1}^n a_{ik}b_{kj}$ for $A,B \in L_n$ yields for $i,j \in \{1,2\}$: \begin{align*}
%(AB)_{11}&= a_{11}b_{11} +a_{12}b_{21},\\
%(AB)_{12}&= a_{11}b_{12} +a_{12}b_{22},\\
%(AB)_{21}&= a_{21}b_{11} +a_{22}b_{21},\\
%(AB)_{22}&= a_{21}b_{12} +a_{22}b_{22}.
%\end{align*}
Obviously the map  $R: L_r \rightarrow GL_2(\C)$ defined by  $(a_{ij}) \mapsto \left[\begin{array}{cc} a_{11}&a_{12} \\ a_{21} & a_{22} \end{array}\right]$ is a group homomorphism.
Clearly, $A$ and $B$ are in $L_r$. Also, by the assumptions,  
$R(A)=  \left[\begin{array}{cc}1&0 \\ A_{21}&\zeta_n \end{array}\right] $ 
and $R(B) =\left[\begin{array}{cc}1&d \\ 0&\zeta_m \end{array}\right] $.
The assumption  $|A_{21}d| \geq 4|z_{k_n}z_{k_m}|$ and \Cref{Sanov_cyclic}  yield that 
$\langle R(A),R(B) \rangle \cong C_{i_n} \star C_{i_m}$.
 Since $A$ has order $i_n$ and $B$ has order $i_m$ we obtain that $<A, B> \cong C_{i_n} \star C_{i_m}$.
\end{proof}

\section{Solvable subgroups and Free sub-semigroups}

In this section we first construct subgroups of $\U (\Z G)$ that are  abelian-by-finite and solvable of length $2$.
We will give explicit generators for such groups; the construction of these generators  is  a generalisation of the  bicylic units (seen from another perspective, they are products of a bicylic unit with a trivial unit).
Modifying the bicyclic units differently, we next will construct solvable-by-finite subgroups of $\U (\Z G)$ that are 
not nilpotent-by-finite. 
By the well known result of Rosenblatt \cite{Ros} such groups  contain a free submonoid of rank $2$. 
Actually, we  give explicit generators of such monoids.

We begin by recalling  some definitions and notations (we use the same notations as in \cite{EricAngel1}). 
The bicyclic  units in $\Z G$ are the unipotent units of the type
  $$b(g,\tilde{h}) = 1+(1-h)g\tilde{h} \quad \mbox{ and  } \quad b(\tilde{h},g) =1+\tilde{h}g(1-h),$$
where $g,h\in G$. These units are non-trivial(i.e. they do not belong to $G$) if $g\not\in N_{G}(\langle h\rangle)$, the normalizer of $\langle h \rangle$ in $G$. The Bass units are the units of the  type
 $$u_{k,m}(g) = (1+ g + \cdots g^{k-1})^{m} + \frac{1-k^{m}}{o(g)}\tilde{g},$$
where  $g\in G$  and $k,m$ are  positive integers such that $k^{m}\equiv 1 \mbox{ mod } o(g)$, and $1\leq k< n$.
In \cite{bob-free} V. Bovdi introduced the following beautiful new units in $\Z G$ by modifying the bicyclic unit construction and as such constructed units that are often periodic units:
  $$b_k(g,\tilde{h}) = h^k + (1-h) g \widetilde{h} \quad \mbox{ and }  \quad
  b_k(\tilde{h},g) = h^k + \widetilde{h} g (1-h),$$ 
with  $g,h\in G$ and $k$  a positive integer.  Clearly $b_{o(h)}(g,\tilde{h})=b(g,\tilde{h})$,
$b_{o(h)}(\tilde{h},g)=b(\tilde{h},g)$, $b_{k}(g,\tilde{h})=h^{k} b(h^{-k}g,\tilde{h})$ and $b_{k}(\tilde{h},g)=h^{k}b(\tilde{h},h^{-k}g) =h^{k}b(\tilde{h},g)$.
So, again these units are non-trivial precisely when $g$ does not belong to  $N_G(\langle h\rangle )$.
We call these units the \emph{Bovdi units}.  
Also, as remarked in Problem 4 in \cite{JesMarNebKim}, if a unit $b_k(g,\tilde{h}) = h^k + (h-1) g \widetilde{h}$ (respectively $b_k(\tilde{h},g) = h^k + \widetilde{h} g (h-1)$) 
is of finite order then it is  rationally conjugate to $h^{k}$. This easily follows from Lemma 37.6 in \cite{Seh-book}, which says that two elements of $\mathcal{U}(\Z G)$ of the same finite order are rationally conjugate if all irreducible representations of $G$ coincide on both elements.
We recall a lemma from \cite{bob-free} where it is shown when a Bovdi unit is of finite order.  For completeness' sake a proof is included.

\begin{lemma}\label{lemma_order}
Let   $g,h \in G$ and suppose that  $g \not\in N_{G}( \langle h \rangle)$.
Suppose $m$ is the smallest positive integer  such that $g \in N_G( \langle h^m \rangle )$.
For an integer $1 \leq k < o(h)$, the elements $b_k(g,\tilde{h})$ and $b_k(\tilde{h},g)$ are units in $\Z G$ and \begin{enumerate}
\item if $(k,m)=1$, then $o(b_k(g,\tilde{h}))= o(b_k(\tilde{h},g)) = o(h^k)$,
\item else both $b_k(g,\tilde{h})$ and $b_k(\tilde{h},g)$ have infinite order.
\end{enumerate}
\end{lemma}
\begin{proof}
Remark that for a positive integer $t$ we have 
 $$b_k(g,\tilde{h})^t = (h^k)^t + \left(\sum_{i=0}^{t-1} (h^k)^i\right) (1-h) g \tilde{h}.$$ 
 %Put $\alpha_t = \left(\sum_{i=0}^{t-1} (h^k)^i\right) (1-h) g \tilde{h}$. 
 Because  $g \not\in \langle h \rangle$, we have that 
$(h^k)^t \notin \Supp ( \left(\sum_{i=0}^{t-1} (h^k)^i\right) (1-h) g \tilde{h})$ and thus  $(h^k)^t \in \Supp (b_k(g,\tilde{h})^t).$
In particular, if $b_k(g,\tilde{h})$ is torsion, then $(h^k)^{o(b_k(g,\tilde{h}))} = 1$ and $o(h^k) \mid o(b_k(g,\tilde{h}))$.

%Now, consider the sets $\{ ik \mid 0 \leq i \leq \frac{o(h)}{(k,o(h))}-1 \}$ and $\{ i (k,m) \mid 0 \leq i \leq \frac{m}{(k,m)}-1\}$, who are equal modulo $m$. Set $t_k = \frac{|\{ ik \mid 0 \leq i \leq \frac{o(h)}{(k,o(h))}-1 \}|}{|\{ i (k,m) \mid 0 \leq i \leq \frac{m}{(k,m)}-1\}|}$. 
Note that the minimality assumption on $m$ implies that    $m | o(h)$. A straightforward calculation then yields 
\begin{eqnarray}
\quad \quad  \quad \quad \quad (b_k(g,\tilde{h}))^{o(h^k)} &=& (h^k)^{o(h^k)} + \left(\sum_{i=0}^{o(h^k)-1} (h^k)^i\right) (1-h) g \tilde{h} ,
 \nonumber \\ 
 &=& 1 + t \left(\sum_{i=0}^{\frac{m}{(k,m)}-1} (h^{(k,m)})^i\right) (1-h) g \tilde{h},  \quad  \quad \quad \quad \quad \quad \quad  \quad \quad \quad\quad (3.1)  \label{tk}
\end{eqnarray}
for some positive integer $t$.

To prove part (1), assume $(k,m)=1$.  Since $g\in N_{G}(\langle h^m \rangle )$, we then get that 
$\left(\sum_{i=0}^{\frac{m}{(k,m)}-1} (h^{(k,m)})^i\right) (1-h) g \tilde{h} = (1-h^m)g\tilde{h} = 0$.
Hence,  $(b_k(g,\tilde{h}))^{o(h^k)} =1$ and thus $o(b_k(g,\tilde{h})) | o(h^{k})$. So, by the above $o(b_k(g,\tilde{h})) =o(h^{k})$.
This proves part (1).

To prove part (2),  assume  $(k,m) \neq 1$. We give a proof by contradiction. So  suppose 
$b_k(g,\tilde{h})$ has finite order. 
By the first part of the proof,  $1=(h^k)^{o(h^k)}  \in \Supp (b_k(g,\tilde{h})^{o(h^k)})$.
The Berman-Higman Theorem therefore yields that  $(b_k(g,\tilde{h}))^{o(h^k)} = 1$ . 
So, by (\ref{tk}),
 $\left(\sum_{i=0}^{\frac{m}{(k,m)}-1} (h^{(k,m)})^i\right) (1-h) g \tilde{h} = 0$. 
Thus,  $g\in N_{G}(\langle h^{(k,m)i+1})$ for some $0\leq i < \frac{m}{(k,m)}$;
in  contradiction with the minimality of $m$.
This proves part (2)  for $b_k(g,\tilde{h})$.
The proof for $b_k(\tilde{h},g)$ is similar.
\end{proof}

\begin{proposition}\label{virtualabelian}
Let $G$ be a finite group and $g,h\in G$.
Let $u=b_k(g,\tilde{h})$ be a Bovdi unit and $w=b(g,\tilde{h})$ a bicyclic unit.
Then, the group $\langle u, u^w \rangle$ is abelian-by-finite and solvable of length at most $2$. 
Similarly, if $b_k (\tilde{h},g)$ is a Bovdi unit and $v= b(\tilde{h},g)$ then  $\langle b_k (\tilde{h},g), (b_{k} (\tilde{h},g))^{v}\rangle$ is an abelian-by-finite (and solvable) group.
\end{proposition}

\begin{proof}
We prove the first part. Put $n = o(h)$.
Write the Pierce decomposition of $\Q G$ with respect to the non-central idempotent $e=\widehat{h}$ in a natural matrix form as follows:
 $$\Q G = \left[\begin{array}{cc} (1-e)\Q G(1-e) & (1-e)\Q Ge \\ e\Q G (1-e) & e\Q Ge\end{array}\right]. $$
Note that $u = wh^k$ and $u^w = h^kw$, 
$w = \left[\begin{array}{cc} 1-e &  (1-h)nge \\ 0 & e\end{array}\right]$ and 
$h^k= \left[\begin{array}{cc} h^k(1-e) & 0\\ 0 & e\end{array}\right]$.
Hence 
$$u = \left[\begin{array}{cc} h^k(1-e) & (1-h)nge \\ 0 & e\end{array}\right] \quad 
\mbox{ and } \quad u^w = \left[\begin{array}{cc}  h^k(1-e) &  h^k (1-h)nge \\ 0 & e \end{array}\right] .$$
So, $\langle u,u^w \rangle$ is a subgroup of  $\left[\begin{array}{cc}\langle h \rangle (1-e) & (1-e)\Q Ge \\ 0 & e\end{array}\right] = H$. Clearly, this group $H$ contains  the abelian group 
$N = \left[\begin{array}{cc} 1-e & (1-e)\Q Ge  \\ 0&e\end{array}\right]$ 
as a normal subgroup and  $H/N \cong \langle h \rangle$.
So,  $H$, and thus also $\langle u,u^w \rangle$, is an abelian-by-finite group.  The result follows.
%
%For the second statement simply note that $[A,B]$ is a strict upper triangular matrix (in particular is contained in $N$, an abelian group) of infinite order.
\end{proof}

Note that the group constructed in the previous proposition is finite if the Bovdi unit is trivial. Furthermore it is infinite cyclic if $h^k=1$.

The proposition implies that  $\langle u, u^w \rangle$ does not contain a free group or free monoid (since nilpotent-by-finite groups have polynomial growth)  if $u$ is not a trivial Bovdi unit and $h^k \neq 1$. In particular $ \langle u,u^w \rangle \ncong C_t \star C_t$ for $t\geq 3$ or $t=\infty$ since otherwise it would contain $\langle u u^w, u^w u \rangle$, a free group of rank $2$. We point this out due to \cite{bob-free} where it is asserted that $ \langle u,u^w \rangle \cong C_t \star C_t$.
However in the case $o(h)=2$ and $k$ is odd,  then $\langle u, u^w \rangle \cong C_2 \star C_2$. Indeed, if $o(h)=2$ then by \Cref{lemma_order} also $o(u^w)=o(u) = 2$.
Moreover, $u u^w = wh^{2k}w= w^2$ has infinite order. Since $(u^w)^{-1}(u u^w)u^w = u^w u = (u u^w)^{-1}$ we get that  $\langle u, u^w \rangle = \langle u u^w, u^w \rangle \cong \Z \rtimes C_2 \cong C_2 \star C_2$ as claimed.

%\subsection{Free semigroups} \label{free semigrps section}
In order to construct free monoids in $\U (\Z G)$ we generalise  the Bovdi units via the use of an arbitrary   unit in $\Z \langle h\rangle$ (instead of $h^k$).
More precisely, for $g,h\in G$  and $ u(h)\in \U (\Z \langle h\rangle)$ define
 $$b(u(h), g,\tilde{h})= u(h) + (1-h)g\widetilde{h} \quad \mbox{ and } \quad b(u(h),\tilde{h},g)= u(h) +\tilde{h}g(1-h).$$
Clearly,  $b(u(h), g,\tilde{h})=u(h) (1+(1-h) (u(h))^{-1} g \widetilde{h})$ and thus it is a unit. Similarly, $b(u(h),\tilde{h},g)$ is a unit.
 % \quad  \mbox{ and } \quad b_2 + (g-1)h\widetilde{g}$ for some bass units $b_i = \mu_{k_i,m_i}(g)$ and elements $g,h \in G$ such that $h \notin N_G(g)$. 
We now show that the $\Q$-algebra generated by $h$ and $a=(1-h)g\tilde{h}$ is contained in a (generalised) upper triangular matrix algebra.
By a theorem of Perlis and Walker, $\Q\langle h \rangle \cong \prod\limits_{d \mid o(h)} \Q (\zeta_d)$, where $\zeta_d$ denotes a primitive $d^{th}$-root of unity. 
%In particular, $\Q (\zeta_n)$ is a Wedderburn component of $\mathbb{Q}\langle h \rangle$, where $n=o(h)$. 
Thus, there exists a  primitive  idempotent $f \in \Q \langle h \rangle$ such that $\Q \langle h \rangle f \cong \Q (\zeta_n)$, with $n = o(h)$. 
Note that  $a^2 = 0$ and $a u= \omega (u)a$ for any  $u\in \Q \langle h \rangle$.
An easy computation  shows that the map 
\begin{eqnarray}
\quad \quad \quad \quad \quad \quad \quad \quad \quad \quad \quad \quad \varphi \; : \;  \Q \langle h \rangle \oplus \Q \langle h \rangle a
\rightarrow 
       \left[\begin{array}{cc} \Q (\zeta_n) & \Q (\zeta_n) \\ 0 & \Q \end{array} \right], \quad \quad \quad \quad \quad \quad \quad \quad \quad   (3.2) \label{uppertriangular}
\end{eqnarray} 
defined by 
$$u_1 + u_2 a \mapsto \left[\begin{array}{cc} u_1 f & u_2f \\ 0 & \omega(u_1) \end{array} \right].
$$
is a $\Q$-algebra morphism.
Hence, if $u(h)$ is a unit of $\Z \langle h \rangle$ of augmentation $1$ then $\varphi (b(u(h), g,\tilde{h}))$ is a matrix of the type
$\left[ \begin{array}{cc} x&y \\ 0&1 \end{array} \right] $, with $x,y\in \Q (\zeta_n )$.

In order to show that two specific matrices of such a type generate a free monoid, we need, for a field $K$ (for our purposes it is sufficient to consider the case that $K=\C $),  a criterion when two elements of the group of affine transformations of the $K$-line $$\A(K) = \left\lbrace \left[ \begin{array}{cc} x&y \\ 0&1 \end{array} \right] \mid x \in K^{\times}, y \in K \right\rbrace,$$ generate a free monoid. 
This is stated in the   following free monoid variant of the Ping-Pong Lemma (see for example \cite[Lemma $2.1$]{Bre}). 
%We add a proof for the readers convenience.
Note that there are essentially  two types of transformations. Those with   $x = 1$, a translation which has  no fix-points, and those with  $x \neq 1$,  a transformation with fix-point  $\frac{y}{1-x}$ and this can be seen as a homothety about this point with scaling factor $x$.

\begin{lemma}{\emph{(Ping-Pong Lemma for Monoids).}}\label{Lemma_PingPongS}
If $A,B\in \A (\C )$   have respective scaling factors $x$ and $y$ and $|x|, |y|\leq \frac{1}{3}$  then $A$ and $B$ generate a free monoid of rank $2$.
\end{lemma}
%\begin{proof}
%Let $d = |p-q|$. By definition, the open balls $B(p,\frac{d}{2})$ and $B(q,\frac{d}{2})$ are disjoint and by the assumptions $x$ maps both of them into $B(p,\frac{d}{2})$, while $y$ maps both of them into $B(q,\frac{d}{2})$.
%
%Suppose we have two different non-trivial words $w_1 = x^{n_1}y^{m_1}\ldots x^{n_s}y^{n_s}$ and $w_2 = x^{k_1}y^{l_1}\ldots x^{k_t}y^{l_t}$ with non-negative powers that are equal in the semigroup generated by $x$ and $y$. Without loss of generality, we may assume that $n_1 < k_1$ or $m_1 < l_1$ (by repeatably multiplying with suitable inverses of $x$ and $y$ on the left). We consider the case $n_1 < k_1$, but the other case is similar. Now, again by left multiplication with $x^{-n_1}$, we may assume that $n_1 = 0$ while $k_1 > 0$. If we let both words act on the affine line, $w_1$ sends $q$ to a point in $B(q,\frac{d}{2})$ while $w_2$ sends it to a point in $B(p, \frac{d}{2})$, a contradiction.
%\end{proof}

Actually the lemma is valid in the more general context of the affine group  $\A(K)$ over a (non-) archimedean local field $K$. 
In case of a non-archimedean field $K$ the condition on the scaling factors has to be replaced by  $|x|, |y|<1$.

The following result is stated for finite groups. It is clear from the proof that it remains valid for arbitrary groups $G$ provided $H$ is a finite subgroup of $G$.

\begin{theorem}\label{main theorem free semigroups}
Let  $H$ be a subgroup of   a finite  group $G$. Suppose  $h \in H$ with $o(h)=n$ and
assume 
$\alpha \in \mathbb{Z}G$ is such that $s = 1 +(1-h)\alpha \widetilde{H} \neq 1$.   
Let $\zeta$ be a primitive $n$-th root of unity.   
Suppose  $b_1 = u_{k_1,m_1}(h) $ and $b_2 = u_{k_2,m_2}(h)$ are two non-trivial Bass units (so, $1 < k_1, k_2 < n-1$, $(k_1,n)=(k_2,n)=1$, $k_1^{m_1} \equiv 1 \mod n$ and $k_2^{m_2} \equiv \mod n$). 
If, for $i \in \{ 1,2\}$,
 $$m_i \geq \log_{\left| \frac{\zeta^{k_i}-1}{\zeta-1}\right|} 3 \quad \mbox{ and } \quad \left(\frac{\zeta^{k_1}-1}{\zeta-1}\right)^{m_1} \neq \left( \frac{\zeta^{k_2}-1}{\zeta-1}\right)^{m_2}, $$ 
 then 
$$\{  b_1 + (1-h)\alpha \widetilde{H}, b_2 + (1-h)\alpha \widetilde{H} \}  \qquad \mbox{ and } \qquad  \{ b_1s, b_2s \}
$$
generate free monoids of rank $2$ that are contained in a solvable group.

In particular, 
$\langle b(u_{k_1,m_1}(h),g,\tilde{h}), b(u_{k_2,m_2}(h),g,\tilde{h}) \rangle$ is a free monoid of rank $2$ if $g\not\in N_{G}(\langle h \rangle)$.

\end{theorem}
\begin{proof}
Put  $a = (1-h)\alpha\widetilde{H}$. Then, $a^{2}=0$ and $a\beta = \omega (\beta)a$ for all $\beta \in \mathbb{Q}H$. Let $f \in \Q \langle h \rangle$ be a primitive  central idempotent such that $\Q \langle h \rangle f \cong \Q (\zeta)$. 
Again, as in (\ref{uppertriangular}) we have a $\Q$-algebra morphism
$\varphi : \Q \langle h \rangle \oplus \Q \langle h \rangle a \rightarrow \left[\begin{array}{cc} \Q (\zeta) & \Q (\zeta) \\ 0 & \Q \end{array} \right]$ with 
$u_1 + u_2 a \mapsto \left[\begin{array}{cc} u_1 f & u_2f \\ 0 & \omega (u_1) \end{array} \right]$.
%Indeed, \begin{align*}
%(u_1 + u_2 a)(u_3+u_4 a) &= u_1u_3 + u_1u_4 a + u_2 a u_3 + u_2 a u_4 a, \\
%&= u_1u_3 + u_1u_4 a + \omega(u_3)u_2 a + \omega(u_4)u_2 a^2, \\
%&= u_1u_3 + (u_1u_4+\omega (u_3)u_2)a,
%\end{align*}
%and 
%\begin{align*}
%\left[\begin{array}{cc} u_1 f & u_2f \\ 0 & \omega(u_1) \end{array} \right]\left[\begin{array}{cc} u_3 f & u_4f \\ 0 & \omega(u_3) \end{array} \right] = \left[\begin{array}{cc} u_1u_3 f & (u_1u_4+\omega(u_3)u_2)f \\ 0 & \omega(u_1u_3) \end{array} \right].
%\end{align*}
Since $b_if = \mu_{k_i,m_i}(\zeta) = \left(\frac{\zeta^{k_i}-1}{\zeta-1}\right)^{m_i}$ and $\omega (b_i) = 1$ we obtain that 
$$\varphi (b_i + a) =  \left[\begin{array}{cc} \left(\frac{\zeta^{k_i}-1}{\zeta-1}\right)^{m_i} & 1 \\ 0 & 1 \end{array} \right]
 \quad \mbox{and} \qquad 
\varphi (b_i s) =  \left[\begin{array}{cc} \left(\frac{\zeta^{k_i}-1}{\zeta-1}\right)^{m_i} & \left(\frac{\zeta^{k_i}-1}{\zeta-1}\right)^{m_i} \\ 0 & 1 \end{array} \right].
$$
%\qquad \mbox{and} \qquad 
%\varphi (b_i s) =  \left[\begin{array}{cc} \left(\frac{\zeta^{k_i}-1}{\zeta-1}\right)^{m_i} & \left(\frac{\zeta^{k_i}-1}{\zeta-1}\right)^{m_i} \\ 0 & 1 \end{array} \right].$$
The inverses of these images are  $\left[\begin{array}{cc} \left(\frac{\zeta^{k_i}-1}{\zeta-1}\right)^{-m_i} & -\left(\frac{\zeta^{k_i}-1}{\zeta-1}\right)^{-m_i} \\ 0 & 1 \end{array} \right]$
and $\left[\begin{array}{cc} \left(\frac{\zeta^{k_i}-1}{\zeta-1}\right)^{-m_i} & -1 \\ 0 & 1 \end{array} \right]$ respectively.
% and $\left[\begin{array}{cc} \left(\frac{\zeta^{k_i}-1}{\zeta-1}\right)^{-m_i} & -1 \\ 0 & 1 \end{array} \right]$.
As elements of $\A(\C)$ they have respective fixed points 
   $f_i = \frac{-\left(\frac{\zeta^{k_i}-1}{\zeta-1}\right)^{-m_i}}{1-\left(\frac{\zeta^{k_i}-1}{\zeta-1}\right)^{-m_i}} = \frac{-1}{\left(\frac{\zeta^{k_i}-1}{\zeta-1}\right)^{m_i}-1}$
   and $g_i = \frac{-1}{1-\left(\frac{\zeta^{k_i}-1}{\zeta-1}\right)^{-m_i}}$.
By assumption, $f_1 \neq f_2$ and $g_1\neq g_2$.
%  $\left(\frac{\zeta^{k_1}-1}{\zeta-1}\right)^{m_1} \neq \left( \frac{\zeta^{k_2}-1}{\zeta-1}\right)^{m_2}$, these fixed points are different. 
Moreover, from the assumption that $m_i \geq \log_{\left| \frac{\zeta^{k_i}-1}{\zeta-1}\right|} 3$, we know that $\left|\frac{\zeta^{k_i}-1}{\zeta-1}\right|^{-m_i} \leq 3^{-1}$ . 
Therefore, by Lemma \ref{Lemma_PingPongS}, 
%{\bf ERIC: BUT to apply the lemma we need to work in a local field!!?? Also the absolute value that we use does this
%correspond with the absolute value in the local field? Lemma 2.2 (a Tits lemma) is maybe of use here: on the embedding of a finitely generated field in a local field?} 
the homotheties $\varphi (b_1 + a)^{-1}$ and $\varphi (b_2 + a)^{-1}$, respectively $\varphi(b_1 s)^{-1}$ and $\varphi(b_2s)^{-1}$, 
generate a free monoid (and thus also $\varphi (b_1 + a)$ and $\varphi (b_2 + a)$, respectively $\varphi(b_1 s)$ and $\varphi(b_2s)$).
Hence,  $\{ b_1 + a, b_2 + a\} $ and $\{  b_1 s , b_2 s \} $   generate free monoids as well.
Note that these monoids are contained in a solvable group.
% and thus they can not be contained in a free group.
%However, $\langle b_1 + a, b_2 + a\rangle$ and $\langle b_1 s , b_1 s \rangle$ can not be free as a group. Indeed, if we look at the algebra embedding \begin{align*}
%\phi : \Q \langle g \rangle \oplus \Q \langle g \rangle a &\hookrightarrow \left[\begin{array}{cc} \Q \langle g \rangle & \Q \langle g \rangle \\ 0 & \Q \end{array} \right], \\
%u_1 + u_2 a &\mapsto \left[\begin{array}{cc} u_1 & u_2 \\ 0 & \omega(u_1) \end{array} \right],
%\end{align*}
%then $\langle b_1 + a, b_2 + a\rangle$ and $\langle b_1 s , b_2 s \rangle $ are isomorphic to a subgroup of the finitely generated, meta-abelian (and thus solvable) group $\left[\begin{array}{cc} \Z \langle g \rangle^{\times} & \Z \langle g \rangle^{\times} \\ 0 & 1 \end{array} \right]$. Thus they are also solvable groups and therefore satisfy a group identity, hence it can not be free as a group.
\end{proof}

We include a couple of remarks on the assumptions concerning $m_1$ and $m_2$.
The restrictions on $m_i$ are clearly satisfied if one takes  $k_1 \neq k_2$ arbitrary (and keeping the requirement that  the Bass units are non-trivial) and $m_1=m_2=m$ large enough. The latter is allowed as one may replace the Bass units by powers because  $u_{k,m}(g) u_{k, n}(g) = u_{k, m+n}(g)$.

Also note that for the  restriction $m_i \geq \log_{\left| \frac{\zeta^{k_i}-1}{\zeta-1}\right|} 3$ to  be fulfilled it is necessary that  $\left| \frac{\zeta^{k_i}-1}{\zeta-1}\right| \neq 1$. Because of \Cref{Lemma_rootsofunity} this is equivalent with  $k_i \neq \pm 1 \mod |g|$ (i.e. the Bass units constructed are  non-trivial).

\section{Free product of cyclic groups}

In this section we will give an explicit construction via Bovdi units of a free product $C_p \star C_p$ in the unit group of the integral
group ring of a finite nilpotent group $G$. We first deal with groups of class two. In this case we prove a more general statement.

\begin{theorem}\label{class2main}
Let $G$ be a finite nilpotent group of  class $2$ and let  $g,h \in G$. 
Assume   $o(h)= p^n$, with $p$ a prime number,   and $g \notin N_G(\langle h^{p^i} \rangle)$ for all $0 \leq i < n$. Then,
for any $1 \leq l,t \leq p^n$,
   $$\langle b_l(g,\widetilde{h}), b_t(\widetilde{h},g^{-1})\rangle \cong C_{n_l}\star C_{n_t} \cong \langle b_l(g,\widetilde{h}), b_t(g,\widetilde{h})^{\star}\rangle,$$ 
a free product of cyclic groups, where $n_l = o( b_l(g,\widetilde{h}))$ and $n_t =o(b_t(\widetilde{h},g^{-1}))$ (see \Cref{lemma_order}).
\end{theorem}
\begin{proof}
Without loss of generality we may assume that $G = \langle g,h \rangle$. 
Since $G$ has nilpotency class two, $c = [g,h^{-1}]$ is central and has $p$-power order, say $p^m$ with $m\leq n$ (and similarly $o(c) \leq o(g)$).
Clearly,   $(h^{p^m})^g= c^{p^m}h^{p^m} = h^{p^m}$ and thus  $g \in N_G(\langle h^{p^m}\rangle)$.
The assumptions therefore imply that $n\leq m$ and thus $n=m$.
So $o(h)=o(c)$ and hence the normalizer assumption on $g$ yields that $\langle h \rangle \cap \langle c \rangle =\{ 1 \}$.  So $\langle h,c \rangle = \langle h \rangle \times \langle c \rangle$. 

We will now construct a concrete irreducible complex representation of $G$ so that the proposed Bovdi units will be represented by matrices to which \Cref{prop: generalised Sanov_cyclic} can be applied.
To do so, first note that if $g^i h^j \in \mathcal{Z}(G)$ then 
$1 = [g, g^ih^j] = [g,g^i][g,h^j] = [g,h]^j= c^{-j}$ and 
$1 = [h, g^ih^j] = [h,g^i]^{h^j}[h,h^j] = [h,g]^i = c^{i}$.
Because $o(h)=o(c)$ this implies that  $\mathcal{Z}(G) = \langle g^{o(c)}, c \rangle$ and
$G/ \mathcal{Z}(G) = \langle \overline{g} \rangle \times \langle \overline{h} \rangle \cong C_{o(c)} \times C_{o(c)}$.

Since $\mathcal{Z}(G)$ is a $2$-generated abelian group with $o(c)$ a $p$-power, there is an isomorphism   
  $\varphi : \mathcal{Z}(G) \rightarrow C_{p^{\alpha}} \times C_{k}$,
for some non-negative  integers $\alpha$ and $k$.
Without loss of generality, we may choose $\alpha$ and $k$ such that $\varphi(c) = (c_{\alpha}, c_{\beta})$ and  $o(c) = o(c_{\alpha})$. 
Set $K = \varphi^{-1}(1\times C_{k})$, a subgroup of  $\mathcal{Z}(G)$. 
Then $\langle c \rangle$ embeds faithfully in $\mathcal{Z}(G) / K \cong C_{p^{\alpha}}$. 
Let $H$ be a subgroup of $\mathcal{Z}(G)$ that is minimal for properly  containing $K$.
It is well-known and easily verified that $\varepsilon =\hat{K}-\hat{H}$ is a primitive idempotent of
$\Q \mathcal{Z}(G)$ and $\Q \mathcal{Z}(G)\varepsilon \cong \Q(\zeta_{p^{\alpha}})$ (see for example Lemma 3.3.2 in \cite{EricAngel1}).
Under this isomorphism $c$ is mapped onto the primitive $p^{o(c)}$-th root of unity  $\zeta_{p^{o(c)}}$.

We next show that $\Q G\varepsilon$ is a simple algebra. This follows at once from the description of primitive central idempotents
in strongly monomial groups, for example nilpotent groups (Section 3.5 in \cite{EricAngel1}). However, in this case this can be shown easily.
%For this it is sufficient to show that $\mathcal{Z} (\Q G\varepsilon) =\Q ( \zeta_{p^{o(c)}})$.
To do so, write $\Q G\varepsilon = \sum_{t\in T} \Q\mathcal{Z}(G)\varepsilon t$, with $T$ a transversal of 
$\mathcal{Z}(G)$ in $G$. Let $z=\sum_{t\in T} \alpha_t \varepsilon t \in \mathcal{Z}( \Q G \varepsilon )$, 
with each   $\alpha_t \in \Q \mathcal{Z}(G)\varepsilon$. The centrality of $z$ means that   
$\sum_{t\in T} \alpha_t \varepsilon t= (\sum_{t\in T} \alpha_t t)^{s} =\sum_{t\in T} (\alpha_t t)^s =\sum_{t\in T} \alpha_t t^s=
\sum_{t\in T} \alpha_t [s,t^{-1}]t$ for all $s\in T$. 
Since $[s,t^{-1}]\in \mathcal{Z}(G)$ and  because $\Q \mathcal{Z}(G)\varepsilon$ is a field, we obtain that
$\alpha_{t}t\varepsilon = \alpha_{t}[s,t^{-1}]t$ for all $s,t\in T$. 
If $t\not\in \mathcal{Z}(G)$ then there exists $s\in T$ so that $\varepsilon \neq [s,t^{-1}]\varepsilon$. Hence, 
 $\alpha_t =0$ if $t\not\in \mathcal{Z}(G)$.
So,  $\mathcal{Z}(\Q G\varepsilon) = \Q \mathcal{Z}(G)\varepsilon \cong \Q(\zeta_{p^{\alpha}})$, a field.
Thus the semisimple algebra $\Q G\varepsilon$ indeed is simple and has centre isomorphic to $\Q(\zeta_{p^{\alpha}})$.

Because $[G:\mathcal{Z}(G)] = o(c)^2$ we get that  $\Q G\varepsilon =M_{o(c)}( \Q(\zeta_{p^{\alpha}}))$, a matrix ring of  degree $o(c)$. In order to describe explicitly this representation, we first note that the following non-zero elements 
 $$E_{ii} = \widehat{h^{g^{i-1}}} \varepsilon ,$$
 with $1 \leq i \leq o(c)$, form a complete set of orthogonal primitive idempotents of this matrix ring.
For this it is sufficient to show the orthogonality and this follows from the fact that if  $1 \leq j , i \leq o(c)$ with $i\neq j$ then $o(c^{i-j}) | o(c^{i-1}h)=o(c)$,  $\widehat{c^{i-j}} \varepsilon = 0$ and thus 
\begin{eqnarray*}
E_{ii} E_{jj} &= &\widehat{h^{g^{i-1}}} \varepsilon \; \widehat{h^{g^{j-1}}} \varepsilon = \widehat{c^{i-1}h}\;  \widehat{c^{j-1}h}\varepsilon \;
= \; \widehat{c^{i-j}(c^{j-1}h)} \;\; \widehat{c^{j-1}h} \varepsilon \\
 &=& \frac{1}{o(c)}\sum_{k=0}^{o(c)-1} ((c^{i-j})^{k} (c^{j-1}h)^{k}) \;\; \widehat{c^{j-1}h} \varepsilon \\
 &=& \frac{1}{o(c)} \sum_{k=0}^{o(c)-1} (c^{i-j})^{k} \;\; \widehat{c^{j-1}h} \varepsilon\\
&=&\widehat{c^{i-j}}\varepsilon  \; \widehat{c^{j-1}h} \varepsilon \\
 &=&    0.
\end{eqnarray*}
We  remark some arithmetic concerning these idempotents.
\begin{enumerate}
\item $hE_{ii} = E_{ii} h$ and  $hE_{11} = E_{11}$.
\item $E_{ii}g = \widehat{h^{g^{i-1}}} \varepsilon g = g \widehat{h^{g^{i}}} \varepsilon = g E_{i+1, i+1}$, where the indices are taken modulo $o(c)$.
\item $h E_{ii} = hg^{-(i-1)}g^{i-1}E_{ii} =g^{-(i-1)}c^{i-1}hE_{11}g^{i-1} = c^{i-1}g^{-(i-1)}E_{11} g^{i-1}= c^{i-1}E_{ii}$.
\end{enumerate}

Now, for any $1\leq i,j \leq o(c)$, put
  $$E_{ij}=E_{ii} g^{j-i}E_{jj}.$$
Then, $E_{ij}E_{ji} = E_{ii} g^{j-i} E_{jj} E_{jj} g^{i-j} E_{ii} = E_{ii} g^{j-i}g^{i-j}E_{ii} = E_{ii}$.
Hence, $\{ E_{ij} \mid 1\leq i,j \leq o(c)\}$ is a complete set of matrix units of the matrix ring $\Q G\varepsilon =M_{o(c)}( \Q(\zeta_{p^{\alpha}}))$.

%Now look at the subgroups $K = \langle b, a^{p^m} \rangle$ and $H = \langle b, a^{p^m}, c \rangle$. According to Jespers and del R\'io $3.5.11$, for a pair of subgroups $(H,K)$ of $G$ to be a strong Shoda pair it suffices that $K \trianglelefteq H \trianglelefteq G$ and $H/K$ is a cyclic and maximal abelian subgroup of $N_G(K)/K$. Since $N_G(K) = N_G(b) = H$ (the only difference between $\langle b \rangle$ and $K$ is a central element) is an abelian group, $(H,K)$ forms a strong Shoda pair and their associated idempotent is $\epsilon (H,K) = \frac{\widehat{b}}{|b|}\frac{\widehat{a^{p^m}}}{|a^{p^m}|}(1-\frac{\widehat{c^{p^{x-1}}}}{p})$, where $x$ is the smallest integer such that $c^{p^{x}} \in K$. Following Theorem $3.5.5$ from the same reference, $A_{\Q}(G,H,K) \cong M_n(\Q(\zeta_h) \star (N_G(K)/H))$, with $n = [G:N] = p^y >1$ (for some $y\in \N$) and $h = [H:K] = p^x > 1$. So $A_{\Q}(G,H,K) \cong M_{p^y}(\Q(\zeta_{p^x})) =: M$. Consider the transversal $T = \{ a^i \mid 0\leq i \leq p^y-1\}$ of $N_G(K) = H$ in $G$ and associate with this idempotents $f_i = \epsilon(H,K)^{a^{i-1}} = \epsilon(H,K^{a^{i-1}})$, for $i = 1, \ldots, p^y$. From the characterisation of being a strong Shoda pair, it follows that the $f_i$ form an orthogonal set of $p^y$ elements, thus they correspond the diagonal elements $E_{ii}$ in $M$.

With respect to these matrix units we will now represent the elements 
$u=b_l(g,\widetilde{h})\varepsilon$ and $v=b_t(\widetilde{h},g^{-1})\varepsilon$ as matrices (the calculations for $\langle b_l(g,\widetilde{h}), b_t(g,\widetilde{h})^{\star}\rangle$ are similar, but are left to the reader as an exercise).
The  $(i,j)$-th position in the matrix of the representation of $u$ is determined as follows:
$$E_{ii} uE_{jj} = E_{ii}(h^l + (1-h)g\widetilde{h})E_{jj} = h^l E_{ii} E_{jj} + (1-h) E_{ii} g \widetilde{h}E_{jj}
= h^l E_{ii} E_{jj} + (1-h)g\widetilde{h}E_{i+1,i+1}E_{jj}. $$
So, if $i = j$  then  $E_{ii} u E_{ii} = h^l E_{ii}$ and thus   $E_{11} u E_{11} = h^l E_{11} = E_{11}$. 
On the other hand, if $i +1 = j$ then $E_{ii} u E_{jj} = (1-h)g\widetilde{h}E_{jj} = o(h)(1-h)g \widehat{h} E_{jj} = o(h)(1-h)g E_{11} E_{jj}$, which is zero if $j \neq 1$. If $i \neq j$ and $i+1 \neq j$, then $E_{ii} u E_{jj} = 0$. 
This shows that the non-diagonal entries of the  matrix representation of $u$  are  $0$, 
%except for the $(1,1)$-position, which  equals $1$,   and  the 
except the $(o(c),1)$-position. The value at this position is determined as follows by  the previous:
\begin{eqnarray*}
E_{o(c),o(c)} u E_{11} &=& o(h) (1-h) g E_{11} \; = \; o(h) (1-h) E_{o(c),o(c)}gE_{11}\\ 
%&  =& o(h) (E_{o(c),o(c)}- hE_{o(c),o(c)}) gE_{11}  \\
&=& o(h)(1-c^{-1})E_{o(c),o(c)}gE_{11}\\
& =& o(h)(1-c^{-1}) E_{o(c),1}.
\end{eqnarray*}

A similar calculation and reasoning yields for $v$: $$E_{ii} v E_{jj} = h^tE_{ii} E_{jj} + \widetilde{h}g^{-1}(1-h)E_{i-1,i-1}E_{jj}, $$ which is zero everywhere except on the  $(1, o(c))$-the position and on the diagonal. Explicitly $$E_{11}vE_{o(c),o(c)} = o(h) (1-c^{-1}) E_{1,o(c)}.$$

Through the explicit morphism to $M_{o(c)}(\Q(\zeta_{p^{\alpha}}))$ the elements $u$ and $v$ 
have $o(h) (1-\zeta_{o(c)}^{-1})$ on position $(o(c),1)$ and $(1,o(c))$ respectively since $c$ is mapped to 
$\zeta_{o(c)}$. 
Moreover, on position $(o(c),o(c))$ $u$ has value  $\zeta_{o(c)}^{-l}$ because $h^lE_{o(c),o(c)} = c^{l(o(c)-1)} E_{o(c),o(c)}$, and $v$ has value $\zeta_{o(c)}^{-t}$. 
Because $|o(h) (1-\zeta_{o(c)}^{-1})| \geq 2$,
these matrices satisfy the conditions for \Cref{prop: generalised Sanov_cyclic}, so $\langle u,v \rangle$ has
$C_{o(u)} \star C_{o(v)}$ as an epimorphic image. Then also $\langle u,v\rangle \cong C_{o(u)} \star C_{o(v)}$.
\end{proof}

The previous result will be used to prove our main result of this section.

\begin{theorem}\label{theorem: result finite nilpotent group}
Let $G$ be a finite nilpotent group, $g,h \in G$ such that $g \not\in N_G(\langle h \rangle)$.
If   $o(h)=p$ (a prime number) and $1\leq k,l\leq p-1$  then 
 $$ \langle b_{k}(g,\tilde{h}) , b_{l}(g,\tilde{h})^{*} \rangle \cong C_p \star C_p \cong \langle b_{k}(g,\tilde{h}) , b_{l}(\tilde{h},g^{-1}) \rangle .$$
Conversely,  if $\mathcal{U}(\Z G)$ contains a subgroup  isomorphic with  $C_p \star C_p$ then there  exist  $g,h\in G$ satisfying the assumptions of the first part of the statement.
\end{theorem}

\begin{proof} 
%Write $G=G_p \times G_{p'}$, where  $G_p$ denotes the Sylow $p$-subgroup of $G$ and $G_{p'}$ is a $p'$-subgroup of $G$. Let  $\overline{g}$ denote the natural image of $g$ in $G_p$. Clearly $\overline{g}\not\in N_{G_{p}}(\langle \overline{h} \rangle)$. Furthermore, because $o(h)=p=\overline{h}$, the natural images of 
%$b_{k}(g,\tilde{h})$, respectively $b_{k}(\tilde{h},g)$ in $\Z G_p$ is $b_{k}(\overline{g},\tilde{\overline{h}})$, respectively
% $b_{k}(\tilde{\overline{h}},\overline{g})$. 
% Hence, if $ \langle b_{k}(\overline{g},\tilde{\overline{h}}) , b_{k}(\overline{g},\tilde{\overline{h}})^{*} \rangle \cong C_p \star C_p \cong \langle b_{k}(\overline{g},\tilde{\overline{h}}) , b_{k}(\tilde{\overline{h}},\overline{g}) \rangle $
% then,   it easily is seen that also $ \langle b_{k}(g,\tilde{h}) , b_{k}(g,\tilde{h})^{*} \rangle \cong C_p \star C_p \cong \langle b_{k}(g,\tilde{h}) , b_{k}(\tilde{h},g) \rangle$, because the generators have order $p$ by 
% \Cref{lemma_order}.
Write $g = g_p x$ with $(o(x),p) = 1$ and let $G_p$ denote  the Sylow $p$-subgroup of $G$. Since $G$ is nilpotent, $g_p \notin N_{G_p}(\langle h \rangle)$.
So, in order to prove the first part of the statement, without loss of generality, we may assume that $G=\langle g,h \rangle$ is a $p$-group. We prove the result by induction on the nilpotency class of $G$. Let $u=b_{k}(g,\tilde{h})$ and $v= b_{k}(\tilde{h},g^{-1})$.
If the class is two then the result holds by Theorem~\ref{class2main}.
So, assume that the class of $G$ is more than $2$.
Let $\overline{G}=G/\mathcal{Z}(G)$ and denote by $\overline{\alpha}$ the natural image of $\alpha\in \Z G$ in $\Z \overline{G}$.
Clearly $\overline{h}$ has order $p$. Also, $\overline{g} \not\in N_{\overline{G}}(\langle \overline{h} \rangle)$. Indeed, otherwise $\langle \overline{h} \rangle$ is a normal subgroup of $\overline{G}$. In particular since $\overline{G}$ is nilpotent, some $\overline{h}^k$ would be central in $\overline{G}=\langle \overline{g}, \overline{h} \rangle = \langle \overline{g}, \overline{h}^k \rangle$. This would imply that $\overline{G}$ is commutative, a contradiction with the fact that $G$ is not of nilpotency degree $2$.
So, the assumptions are inherited in the group $\overline{G}$.
The induction hypothesis yields that $\langle \overline{u},\overline{v}\rangle \cong C_p \star C_p$. Hence, since $u$ and $v$ have order $p$,  also $\langle u,v\rangle \cong C_p \star C_p$.

A similar proof can be given for the statement $\langle b_{k}(g,\tilde{h}) , b_{l}(g,\tilde{h})^{*} \rangle \cong C_p \star C_p$.

For the converse, assume that  all cyclic subgroups of $G$ of order $p$ are normal. In this case, Passman and Gon{\c{c}}alves, in \cite[Lemma $1.3$]{GonPass}, have shown that 
all units of order $p$ in $\mathcal{U}(\Z G)$ are trivial. In particular,  $C_p \star C_p$ would not be a subgroup of $\mathcal{U}(\mathbb{Z} (G))$, a contradiction. 
\end{proof}

We immediately  get the following Marciniak-Seghal type result.

\begin{corollary}
Let $G$ be a finite nilpotent group, $g,h \in G$ such that $g \notin N_G(\langle h \rangle)$ and $o(h)=p\neq 2$. Then 
$\langle b_{k}(g,\tilde{h})  \; b_{l}(g,\tilde{h})^{*} ,   b_{l}(g,\tilde{h})^{*} b_{k}(g,\tilde{h}) \rangle$ is a free group of rank $2$.
\end{corollary}

Also the earlier mentioned  result of Gon{\c{c}}alves and Passman is a consequence in case $G$ is a finite nilpotent group.

\begin{corollary}
Let $G$ be a finite nilpotent group and suppose $h\in G$ is a non-central element of order $p \neq 2$. Then a subgroup isomorphic to $C_{p} \star C_{\infty}$ exists in $\mathcal{U}(\mathbb{Z}G)$ and can be explicitly constructed.  
%which  $\langle u, u^{*} u u^{*} \rangle \cong C_{p} \star C_{\infty}$ where $u = b_k(g,\tilde{h})$ for some $g \notin N_G(\langle h \rangle)$.
Moreover there exists a unit $w\in \U (\Z G)$ such that $\langle h ,w \rangle  \cong C_{p} \star C_{\infty}$.
\end{corollary}

\begin{proof}
Because $h$ is not central and of order $p$, the subgroup  $\langle h\rangle$ is not normal in $G$. Hence there exists $g\in G$ such that $g\not\in N_{G}(\langle h \rangle)$. So, by Theorem~\ref{theorem: result finite nilpotent group},  
$\langle u,u^{*}\rangle \cong \langle u \rangle \star \langle u^{*} \rangle \cong C_{p} \star C_{p}$ with $u=b_{1}(g,\tilde{h})$.
Clearly, $\langle u, u^{*}uu^{*}\rangle \cong \langle u \rangle \star \langle u^{*}uu^{*}\rangle \cong C_{p} \star C_{\infty}$.
As mentioned earlier,  there exists a rational unit $v$ such that $u^v=h$. Since $v^{-1}(\Z G)v$ is a $\Z$-order in $\Q G$,
the unit group of this order and $\U (\Z G)$ are commensurable (see for example \cite[Lemma $4.6.9$]{EricAngel1}). Hence, for some positive integer $m$, we get that  $\left((u^{*}uu^{*})^{v}\right)^m\in \U (\Z G)$. So $\langle u^{v}, \left((u^{*}uu^{*})^{v}\right)^m \rangle = \langle h , \left(u^{*}uu^{*})^{v}\right)^m \rangle \cong C_p \star C_{\infty}$. 
Hence, the result follows.
\end{proof}

%\bibliography{mybibs}{}
%\bibliographystyle{amsplain}
%\def\cprime{$'$} \def\cprime{$'$} \def\cprime{$'$}
%\providecommand{\bysame}{\leavevmode\hbox to3em{\hrulefill}\thinspace}
%\providecommand{\MR}{\relax\ifhmode\unskip\space\fi MR }
%% \MRhref is called by the amsart/book/proc definition of \MR.
%\providecommand{\MRhref}[2]{%
%  \href{http://www.ams.org/mathscinet-getitem?mr=#1}{#2}
%}
%\providecommand{\href}[2]{#2}

\end{document}